\documentclass{amsart}

\usepackage{amssymb,color}
\usepackage{amsfonts}
\usepackage{amsmath}
\usepackage{euscript}
\usepackage{enumerate}
\usepackage{graphics}

\newtheorem{theorem}{Theorem}[section]
\newtheorem{lemma}[theorem]{Lemma}
\newtheorem{note}[theorem]{Note}
\newtheorem{prop}[theorem]{Proposition}
\newtheorem{cor}[theorem]{Corollary}
\newtheorem{exa}[theorem]{Example}

\newtheorem*{Theorem1'}{Theorem 1'}

\theoremstyle{definition}

\theoremstyle{remark}

\numberwithin{equation}{section}

\newcommand \Q{{\mathbb Q}}

\newcommand \C{{\mathbb C}}
\newcommand \ke{{\mathrm {ker}}}
\newcommand \End{{\mathrm {End}}}

\newcommand \dm{{\mathrm {dim}}}
\newcommand \chr{{\mathrm {char}}}

\newcommand \Z{{\mathbb Z}}
\newcommand \GL{{\mathrm {GL}}}

\newcommand \Aut{{\mathrm {Aut}}}
\newcommand \la{{\lambda}}

\newcommand \al{{\alpha}}
\newcommand \bt{{\beta}}

\newcommand \Soc{{\mathrm{Soc}}}
\newcommand \Rad{{\mathrm{Rad}}}

\newcommand \SR{{\mathrm{S}}}
\newcommand \F{{\mathrm{F}}}
\newcommand \T{{\mathrm{T}}}
\newcommand \core{{\mathrm{core}}}

\begin{document}

\title [\small{Groups having a faithful irreducible representation}] {\small{Groups having a faithful irreducible representation}}

\author{Fernando Szechtman}
\address{Department of Mathematics and Statistics, University of Regina}
\email{fernando.szechtman@gmail.com}
\thanks{The author was supported in part by an NSERC discovery grant}





\begin{abstract} We address the problem of finding necessary and sufficient conditions for an arbitrary group, not necessarily
finite, to admit a faithful irreducible representation over
an arbitrary field.
\end{abstract}

\maketitle

\section{Introduction}

We are interested in the problem of finding necessary and
sufficient conditions for a group to have a faithful irreducible
linear representation. Various criteria have been found when the group in question is finite, as  described in~\S\ref{sec:h},
so we will concentrate mainly on infinite groups.

Let $G$ be an arbitrary group. Recall that $\Soc(G)$ is the
subgroup of $G$ generated by all minimal normal subgroups of $G$.
It is perfectly possible for $\Soc(G)$ to be trivial.
We denote by $\SR(G)$, $\T(G)$ and $\F(G)$ the subgroups of $\Soc(G)$ generated by
all minimal normal subgroups of $G$ that are non-abelian, torsion abelian and torsion-free abelian, respectively.
For each prime $p$, let $\T(G)_p$ be the $p$-part of~$\T(G)$. It is a vector space over $F_p$. Let $\Pi(G)$ be the set of all primes $p$ such
that $\T(G)_p$ is non-trivial.

{\color{blue} Throughout the paper we view $\T(G)\F(G)$ as a $\Z
G$-module, with $G$ acting by conjugation. Thus, a minimal normal
subgroup of $G$ contained in $\T(G)\F(G)$ is nothing but an
irreducible $\Z G$-submodule of $\T(G)\F(G)$}.

A normal subgroup $N$ of $G$ is said to be essential if every non-trivial normal subgroup of $G$ intersects $N$ non-trivially.


{\color{blue} With this level of generality, it seems unavoidable
that the desired necessary and sufficient conditions be stated
separately. Our main results are as follows.}

{\color{blue}
\begin{theorem}\label{main} Let $G$ be a group and let $K$ be a
field such that:

(1) $\Soc(G)$ is essential;

(2) If $K$ has prime characteristic $p$ and $M$ is a minimal
normal subgroup of $G$ as well as a non-abelian $p$-group that is
not finitely generated, then $M$ admits a non-trivial
irreducible representation over $K$.

(3) $\chr(K)\notin\Pi(G)$;

(4) $\T(G)$ has a subgroup $S$ such that $\T(G)/S$ is locally
cyclic and $\core_G(S)=1$.

\noindent Then $G$ has a faithful irreducible representation over
$K$.
\end{theorem}
}

{\color{blue} Although condition (1) plays a critical role for us,
it is certainly not necessary, as illustrated by the free group
$F\{x,y\}$ on 2 generators. It has no minimal normal subgroups and
can be
faithfully represented, e.g. via $x\mapsto \left(%
\begin{array}{cc}
  1 & 2 \\
  0 & 1 \\
\end{array}%
\right)$, $y\mapsto \left(%
\begin{array}{cc}
  1 & 0 \\
  2 & 1 \\
\end{array}%
\right)$, as in irreducible subgroup of $\GL_2(\C)$ (cf. \cite[\S
2.1]{Ro}). }

{\color{blue} In connection to condition (2), it is an open
question whether a non-abelian simple $p$-group exists lacking
non-trivial irreducible representations in prime
characteristic~$p$. The answer is known to be negative under the
additional assumption that the group be finitely generated, as
shown in \cite[Theorem 6.3]{Pa2}. Regardless of the outcome of
this question, let $M=M(\Q,\leq, F_p)$ be the McLain group
\cite{M}, where $\leq$ is the usual order on the rational field
$\Q$ and $p$ is a prime. Then $M$ is a minimal normal subgroup of
$G=M\rtimes\Aut(M)$ (in fact, $M=\Soc(G)$) as well as a
non-abelian $p$-group that is not finitely generated. Moreover,
the only irreducible $F_pM$-module is the trivial one (since $M$
is a locally finite $p$-group). Groups like $G$ lie beyond the
scope of Theorem \ref{main}. Perhaps surprisingly,
$M\rtimes\Aut(M)$ does have a faithful irreducible representation
over $F_p$ (see \cite{ST} for details). }

In regards to the necessity of conditions (3) and (4), we have the
following result.

\begin{theorem}\label{mca2} Let $G$ be a group such
that $[G:C_G(N)]<\infty$ for every minimal normal subgroup $N$ of
$G$ contained in $T(G)$. Let $K$ be a field and suppose that $G$
admits a faithful irreducible representation over $K$. Then
$\chr(K)\notin\Pi(G)$ and $\T(G)$ has a subgroup $S$ such that
$\T(G)/S$ is locally cyclic and $\core_G(S)=1$.
\end{theorem}

The most general contributions to this problem have been made by
Tushev. Theorems \ref{main} and \ref{mca2} extend \cite[Theorem
1]{T2}, which reqires that every minimal normal subgroup of $G$ be
finite. Under this assumption, $[G:C_G(N)]$ above is finite,
$\F(G)$ is trivial, every irreducible $\Z G$-submodule of $\T(G)$
is finite, and every non-abelian minimal normal subgroup $M$ of
$G$ is a direct power of a finite non-abelian simple group. In
particular, $M$ is not a $p$-group for any prime $p$, and it
therefore has non-trivial irreducible representations over any
field (cf. Lemma \ref{elr3}).

{\color{blue} The condition $[G:C_G(N)]<\infty$ cannot be removed
with impunity from Theorem~\ref{mca2} as the examples from
\cite{ST} show. }

When $G$ itself is finite \cite[Theorem 1]{T2} or, alternatively,
Theorems \ref{main} and \ref{mca2}, yield a full criterion
({\color{blue} no additional hypotheses are required) .}


In prior work, Tushev \cite{T1} found necessary and sufficient
conditions for a locally polycyclic, solvable group of finite
Pr$\mathrm{\ddot{u}}$fer rank to have a faithful irreducible
representation over an algebraic extension of a finite field. He
recently \cite{T3} extended this line of research to solvable
groups of finite Pr$\mathrm{\ddot{u}}$fer rank over an arbitrary
field.

Also recently, Bekka and de la Harpe \cite{Bd} found a criterion
for a countable group to have faithful irreducible unitary complex
representation, although their sense of irreducibility differs
from the purely algebraic meaning given in this paper.



\section{A review of the finite case}\label{sec:h}



The development of our problem for finite groups makes an
interesting story which is often reported with a certain degree of
inaccuracy in the literature. Moreover, all known criteria,
obviously equivalent to each other, are easily obtained from one
another by means of a straightforward argument that does not
involve groups, their representations or whether they are faithful
or not. This prompted us to review the history of this problem for
finite groups in more detail than usual and to indicate how a
direct translation between the various criteria can be carried
out.

For the sake of our historical review we adopt the following conventions: $G$ stands for an arbitrary finite group and a representation of $G$ means a complex representation, unless otherwise stated.

At the dawn of the twentieth century, a well known and established necessary condition was that the center of $G$ be cyclic.
A partial converse was shown by Fite~\cite{F} as early as 1906. He proved that if $G$ has prime power order or,
more generally, if $G$ is the direct product of such groups (that is, if $G$ is nilpotent),
then the above condition is also sufficient.

The first example of a finite centerless group that admits no
faithful irreducible representation was given by Burnside \cite[Note F]{B} in 1911. It was the semidirect product
\begin{equation}
\label{bu} (C_3\times C_3)\rtimes C_2,
\end{equation}
where $C_2$ acts on $C_3\times C_3$ without non-trivial fixed
points. (Another well-known example of a similar kind, due to
Isaacs \cite[Exercise 2.19]{I} is
\begin{equation}
\label{isa} (C_2\times C_2\times C_2\times C_2)\rtimes C_3,
\end{equation}
where, again, $C_3$ acts on $C_2\times C_2\times C_2\times C_2$
without non-trivial fixed points.)

In \cite[Note F]{B} Burnside found a sufficient condition as well. He showed that if $G$
does not contain two distinct minimal normal subgroups whose
orders are powers of the same prime, then $G$ has a faithful
irreducible representation. This condition is not necessary, a
fact recognized by Burnside. To illustrate this phenomenon, let
$$
G=V\rtimes T,
$$
where $V$ is a vector space of finite dimension $d>1$ over a
finite field $F_q$ with $q>2$ elements, prime characteristic $p$,
and $T$ is the diagonal subgroup of $\GL(V)$ with respect to some
basis of $V$. Clearly, $G$ is the direct product of $d$ copies of
$$
F_q^+\rtimes F_q^*,
$$
which admits a faithful irreducible representation of dimension $q-1$. When
$q>2$ this group has trivial center, so the corresponding tensor
power representation of $G$ is not only irreducible but also
faithful.

Observe that while $G$ has $d>1$ distinct minimal normal subgroups
of order $q$, the normal subgroup they generate, namely $V$, is
clearly a cyclic $F_p G$-module. This is precisely the condition
that Burnside missed.

Since the center of a nilpotent group intersects every non-trivial
normal subgroup non-trivially, Fite's result follows from
Burnside's. The groups $F_q^+\rtimes F_q^*$ used above show that
the converse fails. Because of his work in \cite[Note F]{B} we
will refer to the problem at hand as Burnside's problem or
Burnside's question.

The relevance of $\Soc(G)$ to his problem was already apparent to
Burnside in 1911, although the nature of $\Soc(G)$ was not
sufficiently understood at the time for him to produce a solution.
The structure of $\Soc(G)$ was elucidated by Remak \cite{R} in
1930, and his description has been implicitly or explicitly used
by every author who eventually wrote about Burnside's question.

The first paper addressing Burnside's problem was written by Shoda
\cite{S} in 1930. An error in \cite{S} was quickly pointed by
Akizuki in a private letter to Shoda. This letter included the
first correct solution to Burnside's question. Akizuki's criterion
and a sketch of his proof appeared in a second paper by Shoda
\cite{S2} in 1931, which also included an independent proof of
Akizuki's criterion by Shoda.

Let us outline Akizuki's criterion as it appears in \cite{S2}. Let
$T$ be one of the factors appearing in a decomposition of $\T(G)$
as a direct product of abelian minimal normal subgroups of $G$,
and let~$s$ be the total number of these factors isomorphic to $T$
via an isomorphism that commutes  with the inner automorphisms of
$G$. The group $T$ is isomorphic to the direct product of, say
$r$, copies of $\Z_p$ for some prime $p$. The total number of
endomorphisms of $T$ that commute with all inner automorphisms of
$G$ is of the form $p^g$ for some positive integer $g$. Akizuki's
criterion is that $G$ admits a faithful irreducible representation
if and only if $sg\leq r$, and this holds for all factors $T$
indicated above.

The reader will likely be as baffled by this criterion as subsequent
writers on this subject were. Akizuki's condition seemed difficult to verify in practice, to say the least, and alternative
criteria were sought.

A second criterion was produced by Weisner \cite{W} in 1939.
According to Weisner, $G$ has a faithful irreducible
representation if and only if for every $p\in\Pi(G)$ there exists
a maximal subgroup of $\T(G)_p$ that contains no normal subgroup
of $G$ other than the trivial subgroup.

Shortly after Weisner's paper, Tazawa \cite{T} extended Burnside's problem and
asked under what conditions $G$ would admit a faithful representation with $k$ irreducible
constituents. His answer was given along the same lines as Akizuki's criterion.

The next paper on the subject was written by Nakayama \cite{N} in 1947. He was
the first to consider Burnside's problem over fields of prime characteristic, obtaining
a full criterion which is a slight modification of Weisner's criterion as stated over $\C$.
Nakayama seems to have been unaware of Weisner's prior work. A great deal of \cite{N}
is devoted to solve several generalizations of Burnside's problem.

The following paper on Burnside's question was written Kochend$\mathrm{\ddot{o}}$rffer \cite{K} in 1948.
He gave another proof of Akizuki's criterion and, unaware of Weisner's paper, stated and
proved Weisner's criterion. He wished to produce a condition
that was easily verifiable in practice, and proved that if all
Sylow subgroups of $G$ have cyclic center then $G$ has a faithful
irreducible representation (cf. \cite[Exercise 5.25]{I}). This sufficient condition is an
immediate consequence of the one proved by Burnside in 1911. In a second part of his paper
Kochend$\mathrm{\ddot{o}}$rffer addressed and solved Burnside's question for fields of prime characteristic,
unaware of Nakayama's prior solution.


The next solution to Burnside's problem was given by
Gasch$\mathrm{\ddot{u}}$tz \cite{G} in 1954. According to
Gasch$\mathrm{\ddot{u}}$tz, $G$ has a faithful irreducible
representation if and only if $\T(G)$ is a cyclic $\Z G$-module.

With the benefit of hindsight, let us try to explain {\em
directly} why the criteria of Akizuki, Weisner and Gasch$\mathrm{\ddot{u}}$tz are equivalent to each
other. For this purpose, let us slightly translate each of the above
criteria into a more favorable language.

Since $\T(G)$ is a completely reducible $\Z G$-module, it is
readily seen that $\T(G)$ is a cyclic $\Z G$-module if and only if
$\T(G)_p$ is a cyclic $F_p G$-module for each $p\in\Pi(G)$.

Given $p\in\Pi(G)$, a maximal subgroup, say $M$, of $\T(G)_p$ is
the kernel a non-trivial linear functional $\la:\T(G)_p\to F_p$,
and the largest normal subgroup of $G$ contained in $M$ is
therefore
$$
\underset{g\in G}\bigcap gMg^{-1}=\underset{g\in G}\bigcap \ker({}^g \la),
$$
where ${}^g \la:\T(G)_p\to F_p$ is the linear character defined by
$$
({}^g \la)(v)=\la(g^{-1}vg),\quad v\in \T(G)_p.
$$
Since a proper subspace of the dual space $\T(G)_p^*$ annihilates
a non-zero subspace of $\T(G)_p$, we see that Weisner's condition
is that $\T(G)_p^*$ be a cyclic $F_pG$-module for every
$p\in\Pi(G)$.

It is easily seen (cf. \S\ref{du}) that if $K$ is a field, $A$ is a $K$-algebra with
involution, and $V$ is a finite dimensional completely
reducible $A$-module then $V$ is cyclic if and only if so is its
dual $V^*$. This shows directly that the criteria of Weisner and
Gasch$\mathrm{\ddot{u}}$tz are equivalent.

Let $T$ be an abelian minimal normal subgroup of $G$, that is, an
irreducible $F_p G$-module of $\T(G)_p$ for some $p\in\Pi(G)$. The
quantities $s$, $g$ and $r$ of Akizuki's criterion are
respectively equal to the multiplicity of $T$ in $\T(G)_p$, the
dimension of the field $F=\mathrm{End}_{F_p G}(T)$ over $F_p$, and
the dimension of $T$ itself over $F_p$. Now $T$ is an $F$-vector
space of dimension $r/g$, so Akizuki's criterion is that the
multiplicity of $T$ in $\T(G)_p$ be less than or equal
$\dm_{F}(T)$. The above translation of Akizuki's criterion was
already known to P$\mathrm{\acute{a}}$lfy \cite{P} in 1979, who
gave another proof of Akizuki's criterion, this time over a
splitting field for $G$ of characteristic not dividing $\T(G)$.
(P$\mathrm{\acute{a}}$lfy credited Kochend$\mathrm{\ddot{o}}$rffer
for this result, unaware of Nakayama's prior work.)

It is easily seen (cf. \S\ref{du}) that if $R$ is a left artinian
ring and $V$ is a completely reducible $R$-module then $V$ is
cyclic if and only if for every irreducible submodule $W$ of $V$
the multiplicity of $W$ in $V$ does not exceed $\dm_D W$, where
$D=\mathrm{End}_R W$. This shows directly that the criteria of
Akizuki, Weisner and Gasch$\mathrm{\ddot{u}}$tz are all equivalent
to each other, with generic arguments that involve no groups at
all.

Burnside's question and its various solutions have also appeared in book form. See the books by Huppert \cite{H}, Doerk and Hawkes \cite{DH},
and Berkovich and Zhmud$'$ \cite{BZ}, for instance.
Incidentally, Zhmud$'$ \cite{Zm} found alternative necessary and sufficient conditions for $G$ to admit a faithful representation
with exactly $k$ irreducible representations, a problem previously
considered by Tazawa \cite{T}.


We close this section with an example. Given a finite field
$F_q$ with $q$ elements and a vector space $V$ of finite dimension $d$ over $F_q$
we consider the subgroup
$$
G(d,q)=V\rtimes F_q^*
$$
of the affine group $V\rtimes \GL(V)$ and set
$$
G(q)=F_q^+\rtimes F_q^*.
$$

For any non-trivial linear character $\la:F_q^+\to \C^*$ the
induced character
$$
\chi=\mathrm{ind}_{F_q^+}^{\color{blue}{G(q)}}\la
$$
is readily seen to be a faithful and irreducible of degree $q-1$.
Any hyperplane $P$ of $V$ gives rise a group epimorphism
$$
G(d,q)\to G(q)
$$
with kernel $P$. This produces $(q^d-1)/(q-1)$ different
irreducible characters of $G(d,q)$ of degree $q-1$, each of them
having a hyperplane of $V$ as kernel. On the other hand, $G(d,q)$
has $q-1$ different linear characters with $V$ in the kernel.
Since
$$
\frac{q^d-1}{q-1}\times (q-1)^2+(q-1)\times 1^2=q^d(q-1)=|G(d,q)|,
$$
these are all the irreducible characters of $G(d,q)$. In
particular, $G(d,q)$ has no faithful irreducible characters if
$d>1$. On the other hand, if $q>2$ the center of $G(d,q)$ is
trivial. Thus, to each finite field $F_q$, with $q>2$, there
corresponds an infinite family of finite centerless groups, namely
$G(d,q)$ with $d>1$, having no faithful irreducible characters for
virtually tautological reasons. Moreover, every non-linear irreducible character of
$G(d,q)$ is uniquely determined by its (non-trivial) kernel.
Observe that the groups (\ref{bu}) and (\ref{isa}) are
isomorphic to $G(2,3)$ and $G(2,4)$, respectively.

\section{Cyclic modules and their duals}\label{du}

{\color{blue}
All modules appearing in this paper are assumed to
be left modules.
}

\begin{lemma}\label{dual} Let $R$ be a left artinian ring and let $V$
be a completely reducible $R$-module. Then $V$ is cyclic
if and only if for each irreducible component $W$ of $V$ the multiplicity of $W$ in $V$
is less than or equal than $\dm_D(W)$, where $D=\mathrm{End}_{R}(W)$.
\end{lemma}

\begin{proof} Since there are finitely many isomorphisms types of irreducible $R$-modules, we may reduce to the case when all irreducible submodules of $V$ are isomorphic. Let $W$ be one of them.
Replacing $R$ by its image in $\mathrm{End}(V)$ we may also assume that $J(R)=0$. Thus we may restrict to the case $R=M_n(D)$
and $W=D^{n}$. A cyclic $R$-module is nothing but a quotient of $R$ by a left ideal. This is just a left ideal of $R$, which is isomorphic to the direct sum of at most $n$ copies of~$W$.

\end{proof}

\begin{cor}\label{dual2} Let $K$ be a field and let $A$ be a $K$-algebra with involution. Suppose that $V$ is a completely reducible finite dimensional $A$-module. Then $V$ is cyclic if and only if so is $V^*$.
\end{cor}

\begin{proof} Since every submodule of $V$ is isomorphic to its double dual,
we may reduce to the case when all irreducible submodules of $V$ are isomorphic. Let $W$ be one of them.
Then the transpose map
$$\mathrm{End}_{K}(W)\to\mathrm{End}_{K}(W^*)$$ sends
$D=\mathrm{End}_{A}(W)$ onto $D^0=\mathrm{End}_{A}(W^*)$, the
division algebra opposite to $D$. Since $\dm_K(W)=\dm_K W^*$, it
follows that $\dm_D(W)=\dm_{D^0}(W^*)$. Now apply Lemma
\ref{dual}.
\end{proof}

\begin{cor} Let $K$ be a field and let $A$ be a $K$-algebra with involution. Suppose that $V$ is a finite dimensional $A$-module.
Then $V$ is cyclic if and only if the irreducible components of $V/\Rad(V)$ satisfy the conditions of Lemma \ref{dual},
and $V^*$ is cyclic if and only if the the irreducible components of $\Soc(V)$ satisfy the conditions of Lemma \ref{dual}.
\end{cor}

\begin{proof} Clearly $V$ is cyclic if and only if so is $V/\Rad(V)$. Moreover,
$$
V^*/\Rad(V^*)\cong \Soc(V)^*,
$$
and by Lemma \ref{dual} we know that $\Soc(V)^*$ is cyclic if and only if so is $\Soc(V)$.
\end{proof}

In particular,  if $V/\Rad(V)$ (resp. $\Soc(V)$) is irreducible then $V$ (resp. $V^*$) is cyclic. In the extreme
case when $V$ is uniserial the so is $V^*$ and, moreover, $V$ and $V^*$ are automatically cyclic.

\section{Modules lying over others}\label{sec:modlying}

{
\color{blue}
\begin{lemma}\label{cuqi} Let $K$ be a field. If $U$ and $V$
be $K$-vector spaces with proper subspaces $X$ and $Y$,
respectively, then the subspace $X\otimes V+U\otimes Y$ of
$U\otimes V$ is proper.
\end{lemma}

\begin{proof} There are non-zero linear functionals
$\alpha:U\to K$ and $\beta:V\to K$ such that $\alpha(X)=0$ and $\beta(Y)=0$.
Let $\gamma:U\otimes V\to K$ be the unique linear functional such that $\gamma(u\otimes v)=\alpha(u)\beta(v)$
for all $u\in U$ and $v\in V$. Then $\gamma\neq 0$ and $\gamma(X\otimes V+U\otimes Y)=0$, as required.
\end{proof}
}

\begin{lemma}\label{equis1} Let $K$ be a field, let $G_1,\dots,G_n$ be
groups and set $G=G_1\times\cdots\times G_n$. Let $J_1,\dots,J_n$
be proper left ideals of $KG_1,\dots,KG_n$, respectively. Then the
left ideal of $KG$ generated by $J_1,\dots,J_n$ is proper.
\end{lemma}

\begin{proof}{\color{blue} Arguing by induction, we are reduced to the case
$n=2$, which follows from Lemma \ref{cuqi} by means of the
isomorphism  of $K$-algebras $KG\cong KG_1\otimes KG_2$.
}
\end{proof}

Given a ring $R$, a subring $S$, an $R$-module $V$, and an
$S$-module $U$, we say that $V$ lies over $U$ if $U$ is isomorphic
to an $S$-submdoule of $V$.

\begin{lemma}\label{equis2} Let $(G_i)_{i\in I}$ be a family of groups and let
$G$ be the direct product of them. Let $K$ be a field and for each
$i\in I$ let $V_i$ be an irreducible $KG_i$-module. Then there is
an irreducible $KG$-module $V$ lying over all $V_i$.
\end{lemma}

\begin{proof} For each $i\in I$ we have $V_i\cong KG_i/J_i$, where $J_i$ is a maximal left ideal of~$KG_i$. Suppose the result is false. Then
\begin{equation}
\label{unx} 1=x_1y_1+\cdots+x_n y_n,
\end{equation}
where $x_i\in KG$ and each $y_i$ is in some $J_k$. {\color{blue}
There is a finite subset $L$ of $I$ such that for
$H=\underset{\ell\in L}\Pi G_\ell$ all $J_k$ and $x_i$ above are
inside of $KH$, against Lemma \ref{equis1}.}
\end{proof}

\begin{lemma}\label{elr1}  Let $H$ be a subgroup of a group $G$ and let $K$ be
a field. Let $V$ be an irreducible $KH$-module. Then there exists
an irreducible $KG$-module lying over~$V$.
\end{lemma}

\begin{proof} As shown in \cite[Lemma 6.1.2]{Pa}, a proper left ideal of $KH$ generates a proper left ideal of $KG$.
The result is an immediate consequence of this observation.
\end{proof}

\section{Restricting scalars}\label{sec:resca}

\begin{lemma}\label{elr2} Let $B$ be a
ring with an irreducible $B$-module $V$. Suppose
there is a subring $A$ of $B$ and a family $(u_i)_{i\in I}$ of
units of $B$ such that: $B$ is the sum of all subgroups $Au_i$;
$u_iA=Au_i$ for all $i\in I$; given any $i,j\in I$ there is $k\in
I$ such that $u_iu_jA=u_k A$. Then

(a) $V$ has maximal $A$-submodule if and only if $V$ has a minimal $A$-submodule,
in which case $V$ is a completely reducible $A$-module (homogeneous if all $u_i$
commute elementwise with $A$).

(b) If $I$ is finite then $V$ is a completely reducible $A$-module of length $\leq |I|$.
\end{lemma}

\begin{proof} (a) Suppose first that $V$ has maximal $A$-submodule $M$.
Since $u_iA=Au_i$ for all $i$, it follows that every $u_iM$ is an $A$-submodule of $V$.
But each $u_i$ is a unit, so $u_iM$ is a maximal $A$-submodule of
$V$. Let $N$ be in the intersection of all $u_iM$. Since $u_iu_jA$ is equal to some $u_kA$, we see that $N$ is
invariant under all $u_i$. But $N$ is closed under addition and
$B$ is the sum of all $Au_i$, so $N$ is $B$-invariant. The irreducibility
of $V$ implies that $N=0$. Thus $V$ is isomorphic, as $A$-module,
to a submodule of a completely reducible $A$-module, namely the
direct sum of all $V/u_iM$. In particular, $V$ has an irreducible $A$-submodule.

Suppose next that $V$ has an irreducible $A$-submodule $W$. Then every $u_iW$ is also $A$-irreducible
and $V$ is the sum of all of them. Thus $V$ is the direct sum of sum of them. Removing one summand produces
a maximal $A$-submodule.

(b) In this case $V$ is a finitely generated $A$-module, so it has a maximal $A$-submodule by Zorn's Lemma.
By above, $V$ is an $A$-submodule of a completely reducible $A$-module of length $\leq |I|$.
\end{proof}

\begin{cor}\label{xr2} Let $G$ be a group and let $L/K$ be a finite radical field
extension. Let $V$ be an
irreducible $LG$-module. Then $V$ is a completely reducible
homogeneous $KG$-module. In particular, $V$ has
an irreducible $KG$-submodule.
\end{cor}

\begin{proof} Arguing by induction, we may assume that $L=K[x]$, where $x^n\in K$ for some $n$,
so Lemma \ref{elr2} applies with $B=LG$, $A=KG$ and $u_i=x^i$ for $1\leq i\leq n$.
\end{proof}

\begin{lemma}\label{elr3} Let $G$ be a non-trivial group and let $K$ be a field.
If $K$ has prime characteristic $p$ we assume that $G$ is not a
$p$-group. Then $G$ has a non-trivial irreducible representation
over $K$.
\end{lemma}

\begin{proof} Our hypothesis ensures the existence of $x\in G$ and a non-trivial
linear character $\la:\langle x\rangle\to K[\zeta]$, where $\zeta$ is
a root of unity. By Lemma \ref{elr1} there is an irreducible $G$-module $V$ over $K[\zeta]$
lying over $\la$. By Corollary \ref{xr2} there is an irreducible $KG$-submodule $U$ of $V$.
Since $V=K[\zeta]U$ is not trivial, neither is $U$.
\end{proof}

\begin{note}{\rm Lemma \ref{elr2} is a generalization of \cite[Theorem
7.2.16]{Pa}, which states that an irreducible $KG$-module $V$ has
an irreducible $KN$-submodule, for a normal $N$ subgroup of $G$,
provided $[G:N]$ is finite. Our statement of Lemma \ref{elr2}
was designed to accommodate Corollary \ref{xr2} as well.

Observe that Corollary \ref{xr2} fails if $L/K$ is an arbitrary
field extension, even if $\dm_L(V)=1$. Indeed, let $K$ be an
arbitrary field, $L=K((t))$, $G=U(K[[t]])$ and $V=L$. Then $V$ is
an irreducible $LG$-module of dimension one. However, when viewed
as a $KG$-module, the submodules of $V$ form the following doubly
descending/ascending infinite chain, where $R=K[[t]]$:
$$
V\supset\cdots\supset Rt^{-2}\supset Rt^{-1}\supset R\supset
Rt\supset Rt^2\supset\cdots \supset 0.
$$
}
\end{note}

Given a group $G$, a representation of a subgroup of $G$ is said to be
$G$-faithful if its kernel contains no non-trivial normal subgroups of $G$.

\begin{lemma}\label{tor} Let $G$ be a group and let $K$ be a field.
Let $p$ be a prime different from $\chr(K)$. Suppose $N$ is a
normal subgroup of $G$ contained in $\T(G)_p$ with a
maximal subgroup containing no non-trivial normal subgroup of $G$.
Then $N$ has a $G$-faithful irreducible representation over $K$.
\end{lemma}

\begin{proof} By assumption there exists a one-dimensional
$G$-faithful module $V$ of $N$ over a $K[\zeta]$, where
$\zeta^p=1$. {\color{blue} By Corollary \ref{xr2}}, there is an
irreducible $KN$-submodule $W$ of $V$. As $K[\zeta] W=V$, it is
clear that $W$ is $G$-faithful.
\end{proof}

\section{The torsion-free abelian part of $\Soc(G)$}\label{sec:torfree}

{\color{blue} Recall that a torsion-free abelian
characteristically simple group is divisible and hence a vector
space over $\Q$.}

\begin{lemma}\label{pinf} Let $A$ be a non-trivial torsion-free divisible abelian group, let $K$ be a field
and let $p$ be a prime different from $\chr(K)$. Then there exists
an irreducible representation $R:A\to\GL(V)$ over $K$ such that
$R(A)\cong \Z_{p^\infty}$.
\end{lemma}

\begin{proof} There is a subgroup $S$ of $A$ such that $B=A/S\cong \Z_{p^\infty}$. Let $L$ be a splitting field over $K$
for the family of polynomials $t^{p^i}-1$ for all $i\geq 1$. Then
$G=L^*$ acts $K$-linearly on the $K$-vector space $V=L$ by
multiplication. We have an injective linear character $\la: B\to G$ whose image $\la(B)=C$
consists of all roots of $t^{p^i}-1$, $i\geq 1$, in $L$. Since $L/K$ is algebraic, $K[C]=K(C)=L$.
Thus, the representation $A\to B\to G\to\GL(V)$ satisfies our requirements.
\end{proof}

\begin{lemma}\label{pinf2} Let $G$ be a group and let $K$ be a field. Suppose
$A=\underset{i\in I}\Pi A_i$, where each $A_i$ is an irreducible
$\Z G$-submodule of $\F(G)$ and $|I|\leq \aleph_0$. Then $A$ has a
$G$-faithful irreducible module over $K$.
\end{lemma}

\begin{proof} There is an injection $i\mapsto p_i$, where each $p_i$ is a prime different from $\chr(K)$. Lemma \ref{pinf} ensures the existence,
for each $i\in I$, of an irreducible representation
$R_i:A_i\to\GL(V_i)$ over $K$ such that $R(A_i)\cong
\Z_{p_i^\infty}$. By Lemma \ref{equis2}, there is an irreducible
representation $R:A\to\GL(V)$ lying over all $R_i$. For a fixed
$i\in I$, we have $V=\underset{a\in A}\sum aV_i=\underset{a\in
J}\oplus aV_i$ for some subset $J$ of $A$, so $R(A_i)\cong
\Z_{p_i^\infty}$.

Suppose, if possible, that $A$ contains a non-trivial normal $N$
subgroup of $G$ which acts trivially on ${\color{blue} V}$, and
let $1\neq x\in N$. Then $x=x_{i_1}\cdots x_{i_n}$, where $1\neq
x_{i_k}\in A_{i_k}$. Since $A_{i_1}$ is an irreducible $\Z
G$-module, there is $r\in \Z G$ such that $r\cdot x_{i_1}$ does
not act trivially on $V_{i_1}$. Let $y_{i_k}=r\cdot x_{i_k}\in
A_{i_k}$ for $1\leq k\leq n$. Then $r\cdot x=y_{i_1}\cdots
y_{i_n}\in N$ and $T(y_{i_1})\cdots T(y_{i_n})=1_V$, which is
impossible since these factors commute pairwise, the order of
$T(y_{i_1})$ is a positive power of $p_{i_1}$, and the order of
every $T(y_{i_k})$, $k>1$, is a power of $p_{i_k}$.
\end{proof}

\begin{lemma}\label{pinf3} Let $G$ be a group and let $K$ be a field. Suppose
$A=\underset{i\in I}\Pi A_i$, where each $A_i$ is an irreducible
$\Z G$-submodule of $\F(G)$ and $|I|$ is infinite. Then $A$ has a
$G$-faithful irreducible module over $K$.
\end{lemma}

\begin{proof} Let $L$ be an extension of $K$ that is algebraically closed and satisfies $|L|=|I|$.
Since $L^*$ is a divisible group, we have $L^*=T\times R$, where
$T$ is the torsion subgroup of $L^*$ and $R$ and is torsion-free.
Moreover, $R=\underset{i\in I}\Pi R_i$, where $R_i\cong \Q$.
Relabelling the factors of $A$, we have $A=B\times C$, where
$B=\underset{n\geq 1 }\Pi B_n$ and $C=\underset{i\in I}\Pi C_i$.
Let $p_1,p_2,\dots$ be the list of all primes different from
$\chr(K)$. For each $p_k$, let $T_{k}$ be the $p_k$-part of $T$.

There is a family of linear characters $\la_i:C_i\to L^*$ such
that $\la_i(C_i)=R_i$ for every $i\in I$ and a family of linear
characters $\mu_k:B_k\to L^*$ such that $\mu_i(B_k)=T_k$. Let
$\al:A\to L^*$ be the linear character determined by the $\la_i$
and $\mu_k$, so that $\al(A)=L^*$.

Every $\la_i$ and $\mu_k$ is non-trivial, each $A_i$ is an
irreducible $\Z G$-module, and the $R_i$ and $T_k$ are
independent. We infer {\color{blue} that} $\al$ is $G$-faithful.
The group $G=L^*$ acts via $K$-linear automorphisms on the
$K$-vector space $V=L$ by multiplication. The composition $A\to
L^*\to\GL(V)$ is a linear representation of $A$ on $V$ over $K$.
As such, it is irreducible since $\al(A)=L^*$, and it is
$G$-faithful as so is $\al$.
\end{proof}




\section{Groups with a faithful irreducible representation}\label{sec:weisner2}

{\color{blue}
\begin{theorem}\label{maif} Let $G$ be a group and let $K$ be a
field such that:

(1) $\Soc(G)$ is essential;

(2) If $K$ has prime characteristic $p$ and $M$ is a minimal
normal subgroup of $G$ as well as a non-abelian $p$-group that is
not finitely generated, then $M$ admits a non-trivial
irreducible representation over $K$.

(3) $\chr(K)\notin\Pi(G)$;

(4) $\T(G)$ has a subgroup $S$ such that $\T(G)/S$ is locally
cyclic and $\core_G(S)=1$.

\noindent Then $G$ has a faithful irreducible representation over
$K$.
\end{theorem}
}

{
\color{blue}
\begin{proof} We first show that $\SR(G)$ has a $G$-faithful irreducible module over
$K$.  By \cite[3.3.11]{Ro}, we have $\SR(G)=\underset{i\in I}\Pi S_i$, where each $S_i$ is a non-abelian
minimal normal subgroup of $G$. Let $i\in I$. If it is not the case that $K$ has prime characteristic $p$ and $S_i$
is a $p$-group, then Lemma \ref{elr3} ensures that $S_i$
has a non-trivial irreducible representation over $K$. If $S_i$ is finitely generated then, being perfect,
it has a non-trivial irreducible representation over $K$ by \cite[Theorem
6.3]{Pa2}. If $K$ has prime characteristic $p$ and $S_i$
is a $p$-group that is not finitely generated, then $S_i$ has a non-trivial irreducible representation over $K$ by condition (2).
Thus, in any case, there is a non-trivial irreducible $KS_i$-module, say $V_i$.
By Lemma \ref{equis2}, there is an irreducible $K \SR(G)$-module $V$ lying over
all $V_i$. It then follows from \cite[3.3.12]{Ro} that $V$ is $G$-faithful.

Next note that, by \cite[3.3.11]{Ro}, we have $\F(G)=\underset{i\in I}\Pi A_i$, where each $A_i$ is a torsion-free abelian minimal normal subgroup of $G$. It then follows from Lemmas \ref{pinf2}
and \ref{pinf3} that $\F(G)$ has a $G$-faithful irreducible module $W$ over
$K$.

Finally we show that $\T(G)$ has a $G$-faithful irreducible module over
$K$. Indeed, let $S$ be the subgroup of $\T(G)$ ensured by condition (4). Then
$$
S=\underset{p\in \Pi(G)}\bigoplus S\cap \T(G)_p.
$$
Since $\T(G)/S$ is locally cyclic and $\core_G(S)=1$, we have
$$
[\T(G)_p:S\cap \T(G)_p]=p,\quad p\in\Pi(G).
$$
Obviously, we still have $\core_G(S\cap \T(G)_p)=1$. On the other hand, condition (3) ensures $\chr(K)\notin\Pi(G)$.
It thus follows from Lemma \ref{tor} that, for every $p\in\Pi(G)$, there is a $G$-faithful irreducible module $X_p$
of $\T(G)_p$ over $K$. By Lemma \ref{equis2}, there is an irreducible $K \T(G)$-module $X$ lying over
all $X_p$. It is clear that $X$ must be $G$-faithful.

Since $\Soc(G)=\SR(G)\times \F(G)\times T(G)$, Lemma \ref{equis2} ensures the existence of an irreducible $\Soc(G)$-module $Y$
over $K$ lying over $V$, $W$ and $X$. It follows from \cite[3.3.11 and 3.3.12]{Ro} that $Y$ is $G$-faithful.

By Lemma \ref{elr1}, there is an irreducible $KG$-module $Z$ lying over $Y$. By condition~(1), $\Soc(G)$ is essential,
so $Z$ must be faithful.
\end{proof}
}



\begin{lemma}\label{gulash} Let $N\unlhd G$ be groups, where $N$
is abelian of prime exponent $p$. Let $K$ be a field and suppose
that $V$ admits a faithful irreducible $KG$-module. Let $H$ be any
subgroup {\color{blue} of $C_G(N)$}. Assume that $V$ has an
irreducible $KH$-submodule $W$. Then $\chr(K)\neq p$ and there is
a maximal subgroup $M$ of $N$ such that $\core_G(M)=1$.
\end{lemma}

\begin{proof} Since $W$ is $KH$-irreducible, $D=\End_{KH}(W)$ is a division ring.
Let $L$ be a maximal subfield of $D$ containing $K$. Then $W$ is
an irreducible $LH$-module and $\End_{LH}(W)=L$. Since $N$ is a
central subgroup of $H$, it follows that $N$ acts on $W$ via
scalar operators given a linear character, say $\la:N\to L^*$.

Suppose, if possible, that $\la$ is trivial. Since $N\unlhd G$ and
$V$ is irreducible, we see that $N$ acts trivially on $V$, against
the fact that $G$ acts faithfully on it. Thus, $\la$ is
non-trivial, so $L$ has a root of unity of order $p$, whence
$p\neq \chr(L)=\chr(K)$. Let $M=\ke(\la)$. Then $[N:M]=p$, so $M$
is a maximal subgroup of $N$. Moreover, since $V$ is
$KG$-irreducible, we have $V=\underset{g\in G}\sum gW$, so
$1=\ke_G(V)\supseteq \underset{g\in G}\cap gM g^{-1}$.
\end{proof}

\begin{theorem} Let $G$ be a group such that $[G:C_G(N)]<\infty$
for every minimal normal subgroup $N$ of $G$ contained in $T(G)$.
Let $K$ be a field and suppose that $G$ admits a faithful
irreducible module $V$ over $K$. Then $\chr(K)\notin\Pi(G)$ and
$\T(G)$ has a subgroup $S$ such that $\T(G)/S$ is locally cyclic
and $\core_G(S)=1$.
\end{theorem}

{\color{blue}
\begin{proof} Given $p\in\Pi(G)$, let $N$ be an irreducible $F_p
G$-submodule of $T(G)_p$ with $F_p G$-homogeneous component $M$. Since
$[G:C_G(N)]<\infty$, \cite[Theorem 7.2.16]{Pa} ensures that $V$
has an irreducible $KH$-submodule, where $H=C_G(N)=C_G(M)$. It
follows from Lemma \ref{gulash} that $\chr(K)\neq p$ and there is
linear character $\la_M:M\to K[\zeta]$, $\zeta^p=1$, such that
\begin{equation}
\label{ifin22} \core_G(\ke\la_M)=1.
\end{equation}
Since $T(G)_p$ is the direct product of its
$F_p G$-homogeneous components there is a unique linear character $\la_p:T(G)_p\to K[\zeta]$
extending all $\la_M$. If we set $S_p=\ke \la_p$, then $[T(G)_p:S_p]=p$. A non-trivial normal subgroup of $G$
contained in $T(G)_p$ must intersect at least one $F_p G$-homogeneous component of $T(G)_p$ non-trivially. This
and (\ref{ifin22}) yield $\core_G(S_p)=1$. Since
$T(G)$ is the direct product of all $T(G)_p$, it is now clear that
the direct product $S$ of all $S_p$, satisfies the stated
requirements.
\end{proof}
}

\begin{exa}{\rm Here we extend the example given at the end of \S\ref{sec:h}. Let $F$ be a field of prime characteristic $p$,
let $V$ be an $F$-vector space of dimension $>1$ (not necessarily
finite), and set $G=V\rtimes F^*$. We claim that, provided $F$ is
finite, $G$ has no faithful irreducible representation over any
field. Indeed, if $|F|<\infty$ then $[G:V]<\infty$, so in view of
\cite[Theorem 7.2.16]{Pa} and Lemma \ref{gulash}, we are reduced
to show that every $F_p$-hyerplane of $V$ contains a non-zero
$F$-subspace of $V$. This verification can be further reduced to
the case $\dm_{F}(V)=2$, in which case a routine counting argument
yields the desired result (which fails, in general, for $F$
infinite).}
\end{exa}

{\color{blue} For reference elsewhere, we next show that Mclain's
group (which has no minimal normal subgroups) admits faithful
irreducible representations in every possible characteristic. }

\begin{exa}{\rm Let $M$ be the McLain group \cite{M} defined over a division ring $D$ and let $K$ be any field.
If $\chr(D)=\chr(K)=p$ is prime the only irreducible
representation of $M$ over $K$ is the trivial one. In every other
case, $M$ has a faithful irreducible representation over $K$.

The first assertion follows from the fact that when $\chr(D)=p$ is
prime, $M$ is a locally finite $p$-group. For each $\al<\bt\in\Q$
consider the subgroup $N_{\al,\bt}=\{1+ae_{\al,\bt}\,|\, a\in D\}$
of $M$. Let $I$ be the set of all pairs $(\al,\bt)\in\Q\times\Q$
such that $\al<0$ and $1<\bt$ and let $N$ be the subgroup of $M$
generated by all $N_{\al,\bt}$ with $(\al,\bt)\in I$. Then $N$ is
an essential normal abelian subgroup of $M$, equal to the direct
product of all $N_{\al,\bt}$ with $(\al,\bt)\in I$.

Suppose first $D$ has prime characteristic $p\neq\chr(K)$. For
each $(\al,\bt)\in I$ let $\la_{\al,\bt}:N_{\al,\bt}\to
K[\zeta]^*$, $\zeta^p=1$, be a non-trivial linear character and
let $\la:N\to K[\zeta]^*$ be the unique extension of the
$\la_{\al,\bt}$ to $N$. Then $\la$ is $M$-faithful, for every
non-trivial normal subgroup of $M$ contains at least one (in fact,
infinitely many) $N_{\al,\bt}$ with $(\al,\bt)\in I$. It follows
from Lemma \ref{elr1} that $M$ has a faithful irreducible module
$U$ over $K[\zeta]$. By {\color{blue} Corollary} \ref{xr2}, there
is an irreducible $KM$-submodule $V$ of $U$. Since $K[\zeta]V=U$,
it is clear that $V$ is faithful.

Suppose next that $\chr(D)=0$. Choose any $0\neq x\in D^+$ and any
prime $p\neq\chr(K)$. Then there is a non-trivial linear character
$\mu:\langle x\rangle\to K[\zeta]^*$, where $\zeta^p=1$. By Lemma
\ref{elr1}, there is an irreducible $D^+$-module $U$ over
$K[\zeta]$ lying over~$\mu$. For each $(\al,\bt)\in I$ let
$U_{\al,\bt}$ be the irreducible $N_{\al,\bt}$-module over
$K[\zeta]$ obtained from $U$ via the isomorphisms $D^+\to
N_{\al,\bt}$ given by $a\mapsto 1+ae_{\al,\bt}$. By Lemma
\ref{equis2}, there is an irreducible $N$-module $V$ over
$K[\zeta]$ lying over all $U_{\al,\bt}$. This implies, as before,
that $M$ has a faithful irreducible module over $K$. }
\end{exa}

\section{Necessary and sufficient conditions for nilpotent groups}

\begin{prop} Let $G$ be a nilpotent group. Let $Z$ be the center of $G$ and let $T$ be the torsion subgroup of $Z$.
Then the following conditions are equivalent:

(a) $G$ admits a faithful irreducible representation.

(b) $Z$ is isomorphic to a subgroup of the multiplicative group of
a field.

(c) $Z$ admits a faithful irreducible representation.

(d) $T$ is locally cyclic.

(e) $T$ is a subgroup of $\Q/\Z$.

(f) For each prime $p$, the $p$-part of $T$ is a subgroup of
$\Z_{p^\infty}$.
\end{prop}

\begin{proof} Suppose first that $V$ is a faithful irreducible $G$-module over a field $K$ and
let $D=\End_{KG}(V)$. Let $L$ be a maximal subfield $L$ of $D$
containing $K$. Then $L=\End_{LG}(V)$, which yields an injective
group homomorphism $Z\to L^*$. This shows that (a) implies (b),
which obviously implies (c).

Suppose next that $V$ is a faithful irreducible $Z$-module over a
field $K$. By Lemma \ref{elr1} there is an irreducible $KG$-module
$U$ lying over $V$. Suppose, if possible, that $N$ is a
non-trivial normal subgroup of $G$ acting trivially on $V$. Since
$G$ is nilpotent, $N\cap Z(G)$ is a non-trivial normal subgroup of
$G$ acting trivially on $U$ and hence on~$V$, a contradiction.
Thus (c) implies (a).

Clearly (b) implies (d). The converse was proven by Cohn \cite{C}.
It is well-known that (d),(e) and (f) are equivalent.
\end{proof}




{
\color{blue}
\noindent{\bf Acknowledgments.} I thank B. Gilligan for his help
with German translation as well as R. Guralnick, A. Olshanskiy, D. Passman, A. Tushev and the referee
for useful comments.
}

\end{document}